\documentclass[12pt,a4paper]{amsart} 
\usepackage{inputenc}
\usepackage{latexsym}
\usepackage{pinlabel}          
\usepackage[a4paper , lmargin = {2cm} , rmargin = {2cm} , tmargin = {2.5cm} , bmargin = {2.5cm} ]{geometry}             
\usepackage{graphicx}
\usepackage{amssymb, amsthm}
\usepackage{epstopdf}
\usepackage{romannum}
\usepackage{tikz}
\usepackage{subfig}
\usepackage[arrow, matrix, curve]{xy} 
\usepackage{hyperref}
\usetikzlibrary{arrows}
\usepackage{enumerate}  
 \usepackage{graphicx}
\usepackage[defaultcolor=red]{changes}
\usepackage{tikz-cd}

\theoremstyle{plain}
\newtheorem*{corollaryA}{Corollary A}
\newtheorem*{theoremB}{Theorem B}
\newtheorem*{corollaryC}{Corollary C}
\newtheorem*{Proposition D}{Proposition D}

\newtheorem{theorem}{Theorem}[section]
\newtheorem{StructureTheorem}[theorem]{Structure Theorem for abelian divisible groups}
\newtheorem{lemma}[theorem]{Lemma}

\newtheorem{corollary}[theorem]{Corollary}

\newtheorem{proposition}[theorem]{Proposition}
\theoremstyle{definition}

\newtheorem{remark}[theorem]{Remark}
\newtheorem{definition}[theorem]{Definition}

\definecolor{applegreen}{rgb}{0.55, 0.71, 0.0}
\definecolor{dgreen}{rgb}{0.0, 0.5, 0.0}

\definechangesauthor[name={}, color=green]{D}

\title[]{Automatic continuity for groups whose torsion subgroups are small}

\author{Daniel Keppeler, Philip M\"oller and Olga Varghese}
\date{\today}

\address{Daniel Keppeler \\
Department of Mathematics\\
University of M\"unster\\ 
Einsteinstra\ss e 62\\
48149 M\"unster (Germany)}
\email{danielkeppeler@uni-muenster.de}

\address{Philip M\"oller\\
Department of Mathematics\\
University of M\"unster\\ 
Einsteinstra\ss e 62\\
48149 M\"unster (Germany)}
\email{philip.moeller@uni-muenster.de}

\address{Olga Varghese\\
Department of Mathematics\\
Otto-von-Guericke University of Magdeburg\\ 
Universit\"atsplatz 2\\
39106 Magdeburg (Germany)}
\email{olga.varghese@ovgu.de}

\begin{document}

	\pagenumbering{arabic}
	\begin{abstract}	
	We prove that a group homomorphism $\varphi\colon L\to G$ from a locally compact Hausdorff group $L$ 
into a discrete group $G$ either is continuous, or there exists a normal open subgroup $N\subseteq L$ such that $\varphi(N)$ is a torsion group provided that $G$ does not include $\mathbb{Q}$ or the $p$-adic integers $\mathbb{Z}_p$ or the Pr\"ufer $p$-group $\mathbb{Z}(p^\infty)$ for any prime $p$ as a subgroup, and if the torsion subgroups of $G$ are small in the sense that any torsion subgroup of $G$ is artinian. In particular, if $\varphi$ is surjective and $G$ additionally does not have non-trivial normal torsion subgroups, then $\varphi$ is continuous. 

As an application we obtain results concerning the continuity of group homomorphisms from locally compact Hausdorff groups to many groups from  geometric group theory, in particular to automorphism groups of right-angled Artin groups and to Helly groups.	
	\bigskip
	
\hspace{-0.4cm}{\bf Key words.} \textit{Automatic continuity, locally compact Hausdorff groups, metrically injective groups, Helly groups, automorphism groups of right-angled Artin groups}	
\medskip

\medskip
\hspace{-0.4cm}{\bf 2010 Mathematics Subject Classification.} Primary: 22D05; Secondary: 20F65	
\end{abstract}

\thanks{The work was funded by the Deutsche Forschungsgemeinschaft (DFG, German Research Foundation) under Germany's Excellence Strategy EXC 2044--390685587, Mathematics M\"unster: Dynamics-Geometry-Structure. The first author was supported by (Polish) Narodowe Centrum Nauki, UMO-2018/31/G/ST1/02681. The second and third author were also partially supported by (Polish) Narodowe Centrum Nauki, UMO-2018/31/G/ST1/02681. The second author was
supported by a stipend of the Studienstiftung des deutschen Volkes. The third author was supported by DFG grant VA 1397/2-1. This work is part of the PhD projects of the first and second author.}

	\maketitle
	
	\section{Introduction}	
In the class of locally compact Hausdorff groups ${\bf LCG}$ one has to distinguish between algebraic
morphisms and continuous morphisms. {\em We will always assume that locally compact  groups have the Hausdorff property.} By ${\rm Hom}(L,G)$ we denote the set
of algebraic morphisms, i.e. group homomorphisms that are not necessarily 
continuous, and by ${\rm cHom}(L,G)$ we denote the subset of continuous group homomorphisms. We are interested in conditions on the discrete group 
$G$ such that ${\rm Hom}({\bf LCG},G)={\rm cHom}({\bf LCG},G)$, i.e every algebraic homomorphism $\varphi\colon L\to G$ is continuous for every locally compact group $L$. From a category theory perspective this is the question whether the forgetful functor from ${\bf LCG}$ to ${\bf Grp}$ is surjective.  

Questions concerning automatic continuity of group homomorphisms from locally compact groups  into discrete groups have been studied for many years. A remarkable result obtained by Dudley in \cite{Dudley} says that any group homomorphism from a locally compact  group into a free (abelian) group is continuous. Further results in this direction can be found in \cite{BogopolskiCorson}, \cite{ConnerCorson}, \cite{CorsonAut} and in \cite{CorsonKazachkov}. A characterization in terms of forbidden subgroups of $G$ was obtained in \cite{CorsonVarghese}: ${\rm Hom}({\bf LCG},G)={\rm cHom}({\bf LCG},G)$ if and only if $G$ is torsion-free and does not contain $\mathbb{Q}$ or the $p$-adic integers  $\mathbb{Z}_p$ for any  prime $p$ as a subgroup. 

By definition, a discrete group $G$ is called $lcH$\textit{-slender} if ${\rm Hom}({\bf LCG},G)={\rm cHom}({\bf LCG},G)$. In geometric group theory it is common to investigate virtual properties of groups, hence we call a group $G$  \textit{virtually lcH-slender} if it has a finite index $lcH$-slender subgroup. Using \cite{CorsonVarghese}, we obtain a characterization of  virtually $lcH$-slender groups. 
\begin{corollaryA}
A group $G$ is virtually $lcH$-slender if and only if $G$ is virtually torsion-free and does not include $\mathbb{Q}$ or the p-adic integers $\mathbb{Z}_p$ for any prime $p$ as a subgroup.
\end{corollaryA}
Many groups from geometric group theory are not $lcH$-slender, but they are virtually $lcH$-slender. Some examples of these groups are Coxeter groups and (outer) automorphism groups of right-angled Artin groups.

\vspace{0.7cm}
The main focus of this article is on automatic continuity for surjective group homomorphisms from locally compact groups into discrete groups. We know that any surjective group homomorphism from a locally compact group into $\mathbb{Z}$ is continuous,
but what happens if we replace the group $\mathbb{Z}$ by a slightly bigger group that contains torsion elements, for example the infinite dihedral group $\mathbb{Z}\rtimes\mathbb{Z}/2\mathbb{Z}$?

Let ${\rm Epi}(L,G)$ be the set of surjective group homomorphisms and  ${\rm cEpi}(L,G)$ the subset consisting of continuous surjective group homomorphisms. The question we address here is the following:\\
\textcolor{black}{\rule{\textwidth}{0.07cm}}	 \begin{quote}    
    		 {\em  Under which conditions on the discrete group $G$ does the equality 
     			$${\rm Epi}({\bf LCG},G)={\rm cEpi}({\bf LCG},G)$$ hold? 
     		}
     		\end{quote}
\textcolor{black}{\rule{\textwidth}{0.07cm}}\vspace{0.2cm}		

It was proven by Morris and Nickolas in \cite{MN} that if $G$ is a non-trivial (finite) free product $\mathbin{*}_{i\in I} G_i$ of groups $G_i$, then ${\rm 
Epi}({\bf LCG},\mathbin{*}_{i\in I} G_i)={\rm cEpi}({\bf LCG},\mathbin{*}_{i\in I}G_i)$. Finite free products of groups are special cases of graph products of groups. 
Given a finite simplicial graph $\Gamma=(V, 
E)$ and a collection of groups $\mathcal{G} = \{ G_u \mid u \in V\}$, the \emph{graph product} $G_\Gamma$ is defined as the quotient
$({\ast}_{u\in V} G_u) / \langle \langle [G_v,G_w]\text{ for }\{v,w\}\in E \rangle \rangle$. Kramer and the third author proved in \cite{KramerVarghese} that if the vertex set of $\Gamma$ is not equal to $S\cup \{w\in V\mid \{v,w\}\in E\text{ for all }v\in S\}$ where the subgraph generated by $S$ is 
complete, then ${\rm Epi}({\bf LCG},G_\Gamma)={\rm cEpi}({\bf LCG},G_\Gamma).$ 
Further, the second and third author proved in \cite{MV20} that if $G$ is a subgroup of a CAT$(0)$ group whose torsion subgroups are finite and $G$ does not have non-trivial finite normal subgroups, then ${\rm Epi}({\bf LCG}, G)={\rm cEpi}({\bf LCG}, G)$ by geometric means.

Our main result is the following.
	\begin{theoremB}
	Let $G$ be a discrete group. If 
			\begin{enumerate}
			\item[(i)] $G$ does not include $\mathbb{Q}$ or the $p$-adic integers $\mathbb{Z}_p$ 	for any prime $p$ as a subgroup,
				\item[(ii)] $G$ does not include the Pr\"ufer $p$-group $\mathbb{Z}(p^\infty)$ for any prime $p$ as a subgroup and
				\item[(iii)]  torsion subgroups in $G$ are artinian,
				\end{enumerate}
			then any group homomorphism $\varphi\colon L\to G$ from a locally compact group $L$ to $G$ is continuous, or there exists a normal open subgroup $N\subseteq L$ such that $\varphi(N)$ is a non-trivial torsion group. 	
			
			If additionally
				\begin{enumerate}				
				\item[(iv)]  $G$ does not have non-trivial torsion normal subgroups, 
			\end{enumerate}			
			then ${\rm Epi}({\bf LCG},G)={\rm cEpi}({\bf LCG},G)$. 
	\end{theoremB}
We want to remark that an open normal subgroup $N$ in a non-discrete locally compact group $L$ is large in the sense that $L/N$ is a discrete group. Thus a group homomorphism $\varphi\colon L\to G$ from a locally compact group $L$ to a discrete group $G$ with properties (i)-(iii) of Theorem B is continuous or the image is almost a torsion group. 		
			
Many geometric groups, i.e. groups that admit a geometric action on a metric space with additional geometry are of our interest. Our main result is inspired by the question if similar automatic continuity results as in the case of CAT$(0)$ groups hold for other geometric groups, in particular for metrically injective groups. We call a group $G$ \textit{metrically 
injective} if it acts geometrically on an injective metric space, i. e.  a metric space that is an injective object in the category of metric spaces and $1$-Lipschitz maps. Examples of metrically injective groups include Gromov-hyperbolic groups \cite{Lang}, uniform lattices in ${\rm GL}_n(\mathbb{R})$ \cite{Haettel} and Helly groups \cite{Helly}, in particular CAT$(0)$ cocompactly cubulated groups, finitely presented graphical $C(4)-T(4)$ small cancellation groups, type-preserving uniform lattices in Euclidean buildings of type $\widetilde{C}_n$ \cite{Helly}, uniform lattices in ${\rm GL}_n(\mathbb{Q}_p)$ \cite{Haettel} and  Artin groups of type FC \cite{HuangOsajda}. 

As an application of Theorem B we obtain an automatic continuity result for group homomorphisms from locally compact groups to metrically injective groups and many other groups from geometric group theory.
	\begin{corollaryC}(see Proposition \ref{classG})
	If $G$ is a subgroup of
\begin{enumerate} 
	\item a virtually $lcH$-slender group,
	\item a cocompactly cubulated ${\rm CAT}(0)$ group,
	\item a ${\rm CAT}(0)$ group whose torsion subgroups are artinian, e.g. a Coxeter group,
	\item a Gromov-hyperbolic group,
	\item a metrically injective group whose torsion subgroups are artinian, e.g. a Helly group whose torsion subgroups are artinian,	
	\item  a finitely generated residually finite group whose torsion subgroups are artinian, e.g. the (outer) automorphism group of a right-angled Artin group,
	\item a one-relator group,
	\item a finitely generated linear group in characteristic $0$,
	\item the Higman group,	
\end{enumerate}
then any group homomorphism $\varphi\colon L\to G$ from a locally compact group $L$ is continuous or there exists a normal open subgroup $N\subseteq L$ such that $\varphi(N)$ is a torsion group. 	

If $G$ does not have non-trivial torsion normal subgroups, then ${\rm Epi}({\bf LCG},G)={\rm cEpi}({\bf LCG},G)$. 
	\end{corollaryC}
	Furthermore we show that we can generate new examples from these old ones by taking extensions and  graph products where the defining graph $\Gamma$ is finite. More precisely, let $\mathcal{G}$ denote the class of groups that do not contain $\mathbb{Q}$ or $\mathbb{Z}_p$ or the Pr\"ufer group $\mathbb{Z}(p^\infty)$ for any prime number $p$ and whose torsion subgroups are artinian. 
	\begin{Proposition D}(see Proposition \ref{Extensions} and Proposition \ref{ClassG})
		The class $\mathcal{G}$ is closed under taking extensions and graph products of groups.
	\end{Proposition D}
\subsection*{Structure of the proof of Theorem B}	
Given a locally compact group $L$, the connected component which contains 
the identity, denoted by $L^ \circ$, is a closed normal subgroup. Thus we 
obtain a short exact sequence of locally compact groups $\left\{1\right\}\to L^\circ\to L\to L/L^\circ\to\left\{1\right\}$, where $L/L^\circ $ is a totally disconnected locally compact group. In this way, the study of group homomorphisms from a general locally compact group reduces to studying group homomorphisms from connected locally compact groups and totally disconnected locally compact groups. The world of connected locally compact groups is well understood. Iwasawa's structure Theorem \cite[Thm. 13]{Iwasawa} tells us that the 
puzzle pieces of these groups are compact  groups and the real numbers. On the other hand, any group with the discrete topology is a totally disconnected locally compact group, hence there are no structure theorems for 
this class of groups, but there is an important result by van Dantzig that says that any totally disconnected locally compact group has an open compact subgroup \cite[III \S 4, No. 6]{Dantzig}. 
 Iwasawa's and van Dantzig's Theorems are the key ingredients in the proof of Theorem B.
 
 Given a group homomorphism $\varphi\colon L\to G$ where $G$ satisfies the conditions (i)-(iii),  in the first step, using Iwasawa's structure Theorem, we prove that  $\varphi(L^\circ)$  is a trivial group. In the second step, using van Dantzig's Theorem and the fact that torsion subgroups in $G$ 
are artinian, we prove that for the induced group homomorphism $\overline{\varphi}\colon L/L^\circ\to \varphi(L)$ there exists a compact open subgroup $K\subseteq L/L^\circ$ which is mapped to a minimal torsion normal subgroup. We then pull the subgroup $\overline{\varphi}(K)$  back and show it is indeed open and also mapped to a torsion normal subgroup under $\varphi$. If condition (iv) is satisfied too,  then the subgroup $\overline{\varphi}(K)$ is trivial. Hence the maps $\overline{\varphi}$ and $\varphi$ have open kernels and therefore are continuous. 

\subsection*{Acknowledgment.} 
The authors thank Linus Kramer for helpful comments on a preprint of this paper and Damian Osajda for pointing out the right attributions.  Additionally the authors thank the anonymous referee for pointing out a mistake in the argument and how to get around it via the strengthening of Corollary 2.5 and the introduction of Lemmas 4.1 and 4.2.

The second and third author are grateful for the opportunity to take part in the Mini-Workshop ``Nonpositively Curved Complexes" in Oberwolfach organized by  Damian Osajda, Piotr Przytycki and Petra Schwer that inspired them to start this project. The third author thanks Samuel Corson for introducing her to the Higman group and for giving arguments that this group satisfies the conditions (i)-(iv) of Theorem B.

The second author thanks the Studienstiftung des deutschen Volkes for the support.

\section{Preliminaries}

\subsection{Locally compact groups}

A {\em locally compact group} is a group endowed with a locally compact Hausdorff topology such that the group operations are continuous. We will always assume that locally compact  groups have the Hausdorff property. The content of this section and much more information can be 
found in various books on topological groups, e. g. \cite{Stroppel}. Examples 
of (connected) locally compact groups are $\mathbb{R}$, the circle group $U(1) = \left\{x\in \mathbb{C} \mid | z |= 1\right\}$ and ${\rm SL}_n(\mathbb{R})$.  Every group endowed with the discrete topology is a totally 
disconnected locally compact group. The $p$-adic integers $\mathbb{Z}_p$ and $\prod_{\mathbb{N}}\mathbb{Z}/2\mathbb{Z}$ are examples of non-discrete totally disconnected locally compact groups. Furthermore, the automorphism group of a locally finite connected graph with the topology of pointwise 
convergence is also such a group. 

There are two results on locally compact groups which we need for the proof of Theorem B.

\begin{theorem}
\label{IwasawaVanDantzig}
Let $L$ be a locally compact group.
\begin{enumerate}
\item (Iwasawa's structure Theorem \cite[Thm. 13]{Iwasawa}) If $L$ is connected, then  $L=H_0\cdot\ldots\cdot H_n\cdot K$ where each $H_i$ is isomorphic 
to $\mathbb{R}$ and $K$ is a compact connected group.
\item  (van Dantzig's Theorem \cite[III \S 4, No. 6]{Dantzig}) If $L$ is totally disconnected, then $L$ has a compact open subgroup.
\end{enumerate}
\end{theorem}

\subsection{Torsion and divisible groups}

Recall that a group $G$ is called a \textit{torsion} group if all its elements have finite order and $G$ is called \textit{torsion-free} if every non-trivial element of $G$ is of infinite order.
For a given group $G$ we denote by ${\rm Tor}(G)$ the subset of $G$ consisting of all finite order elements. If $G$ is abelian, the subset ${\rm Tor}(G)$ is a (torsion) subgroup of $G$ and $G/{\rm Tor}(G)$ is a torsion-free group \cite[Thm. 1.2]{Fuchs}.

\begin{lemma}\label{TorsionNormal}
Let $G$ be a group and $N\subseteq G$ a normal torsion-free subgroup. Let $\pi\colon G\to G/N$ be the canonical projection. If $H\subseteq G/N$ is torsion-free, then $\pi^{-1}(H)$ is torsion-free.
\end{lemma}
\begin{proof}
	Suppose there exists an $h\in \pi^{-1}(H)$ with ${\rm ord}(h)<\infty$. Then $\pi(h)$ also has finite order, so $\pi(h)=1\cdot N$. But that means $h\in N$ and since $N$ is torsion-free, we can therefore conclude $h=1$ which means that $\pi^{-1}(H)$ is torsion-free.
\end{proof}

The group of $p$-adic integers $\mathbb{Z}_p$ is a crucial poison group in the characterization of $lcH$-slender groups given in \cite[Thm. 1]{CorsonVarghese}. We recall a definition of this important group: For a prime $p$ we denote the group of \textit{$p$-adic integers} by $\mathbb{Z}_p$, which is defined as the inverse limit of the inverse system $0\longleftarrow \mathbb{Z}/p\mathbb{Z}\longleftarrow\mathbb{Z}/p^2\mathbb{Z}\longleftarrow\dots\longleftarrow\mathbb{Z}/p^n\mathbb{Z}\longleftarrow\dots$ where each of the maps $\mathbb{Z}/p^{n-1}\mathbb{Z}\longleftarrow\mathbb{Z}/p^n\mathbb{Z}$ reduces an integer mod $p^n
$ to its value mod $p^{n-1}$. The group $\mathbb{Z}_p$ is a torsion-free uncountable compact group  \cite[p. 9]{Sullivan}.

For a natural number $n>0$ an element $g\in G$ is called \textit{$n$-divisible} if there exists an element $h$ in $G$ such that $h^n=g$. If every element of $G$ is $n$-divisible, we call $G$ \textit{$n$-divisible}. We call $g\in G$ \textit{divisible} if $g$ is $n$-divisible for all natural numbers $n>0$ and $G$ is said to be \textit{divisible} if every element $g\in G$ is divisible. Examples of divisible groups include $\mathbb{Q}$, $\mathbb{R}$ and compact connected groups \cite[Cor. 2]{Mycielski}, a non-example  is the compact group $\mathbb{Z}_p$, nevertheless it is $q$-divisible for any prime $q\in\mathbb{P}-\left\{p\right\}$ \cite[p. 40]{Fuchs}.

The structure of quotients of $\mathbb{Z}_p$ can be described using the following result. Since we could not find a reference, we give a proof here which is based on \cite{Mathoverflow}.

\begin{lemma}
\label{QuotientZp}
If $H$ is a non-trivial subgroup of $\mathbb{Z}_p$, then $\mathbb{Z}_p/H\cong T\oplus D$, where $T$ is a finite cyclic group of $p$-power order and $D$ is an abelian divisible group.
\end{lemma}
\begin{proof}
First we recall that any non-trivial closed subgroup of  $\mathbb{Z}_p$ is also open \cite[Prop. 7]{RobertsonSchreiber}. Hence, we have the following equivalence:
$$\text{A non-trivial subgroup }B\text{ is open in }\mathbb{Z}_p\text{ iff }B\text{ is closed in } \mathbb{Z}_p.$$  

Now we show that for a non-trivial subgroup $H\subseteq \mathbb{Z}_p$, the quotient $\overline{H}/H$ is an abelian divisible group. Since multiplication by $p$, denoted by $m_p\colon\mathbb{Z}_p\to\mathbb{Z}_p$, $m_p(g)=p\cdot g$ is continuous and $\mathbb{Z}_p$ is compact, the group $p\overline{H}$ is open and furthermore the group  $p\overline{H}+H=\bigcup_{h\in H}(p\overline{H}+h)$ is open and therefore closed and contains $H$. Hence $\overline{H}=p\overline{H}+H$. We obtain
$p\cdot(\overline{H}/H)=(p\cdot\overline{H})/H=(p\overline{H}+H)/H=\overline{H}/H$, thus $\overline{H}/H$ is $p$-divisible. Furthermore, since $\overline{H}$ is closed, it is $q$-divisible for every prime $q\neq p$ \cite[Lemma 6.2.3]{Fuchs}. So $\overline{H}/H$ is $q$-divisible for all primes $q$ and therefore divisible.

It was proven in  \cite[Thm. 4.2.5]{Fuchs} that a divisible subgroup of an abelian group is a direct summand of that group. Thus there exists a group $T$ such that
$$\mathbb{Z}_p/H\cong T\oplus \overline{H}/H.$$

It remains to prove that $T$ is finite cyclic group of $p$-power order. We have
$$T\cong \mathbb{Z}_p/H / \overline{H}/H\cong \mathbb{Z}_p/\overline{H}.$$

Again, the group $\overline{H}$ is an open subgroup and therefore $\bigcup\limits_{g\in \mathbb{Z}_p} g\overline{H}$ is an open cover of  $\mathbb{Z}_p$. Since $\mathbb{Z}_p$ is compact, this cover has to be finite, hence $\mathbb{Z}_p/\overline{H}$ is finite.

Now we consider the group homomorphism $\pi:= \pi_2\circ \pi_1\colon\mathbb{Z}_p\overset{\pi_1}{\rightarrow} \mathbb{Z}_p/H\cong T\oplus D\overset{\pi_2}{\rightarrow} T.$ Since the group $T$ is finite, this group homomorphism is continuous \cite[Thm. 5.9]{Klopsch}. Thus $T$ has to be cyclic, since the group $\mathbb{Z}_p$ has a dense cyclic subgroup. Further, the group $\mathbb{Z}_p$  is $q$-divisible for any prime $q\neq p$, hence $\pi(\mathbb{Z}_p)=T$ is also $q$-divisible for any prime $q\neq p$. Hence $T$ is a cyclic group of $p$-power order. 
\end{proof}

\subsection{Abelian divisible groups}

Some important divisible abelian groups are the Pr\"ufer $p$-groups $\mathbb{Z}(p^\infty)$: Let $p$ be a prime. For each natural number $n$, consider the quotient group $\mathbb{Z}/p^n\mathbb{Z}$ and the embedding $\mathbb{Z}/p^n\mathbb{Z}\to \mathbb{Z}/p^{n+1}\mathbb{Z}$ induced by multiplication by $p$. The direct limit of this system is called the Pr\"ufer $p$-group $\mathbb{Z}(p^\infty)$. 
Each Pr\"ufer $p$-group $\mathbb{Z}(p^\infty)$ is divisible, abelian and every proper subgroup of it is finite and cyclic. Moreover, the complete list of subgroups of a Pr\"ufer $p$-group is given by  $0\subsetneq \mathbb{Z}/p\mathbb{Z}\subsetneq \mathbb{Z}/p^2\mathbb{Z}\subsetneq\dots\subsetneq\mathbb{Z}/p^n\mathbb{Z}\subsetneq\dots\subsetneq \mathbb{Z}(p^\infty)$. For more information regarding these groups we refer to \cite[Chap. 1.3]{Fuchs}).\\

The structure of abelian divisible groups is completely understood, as can be seen by the following theorem, which can be found in \cite[Thm. 4.3.1]{Fuchs}.
\begin{StructureTheorem}
\label{structureTheorem}
An abelian divisible group is the direct sum of groups each of which is
isomorphic either to the additive group $(\mathbb{Q},+)$ of rational numbers or to a Pr\"ufer $p$-group $\mathbb{Z}(p^\infty)$ for a prime $p$.
\end{StructureTheorem}

In particular this can be applied to homomorphisms from  divisible groups (e.g. $\mathbb{R}$) into groups $G$ that do not contain $\mathbb{Q}$ or the Pr\"ufer $p$-group $\mathbb{Z}(p^\infty)$ for any prime $p$:

\begin{corollary}
\label{Rtrivial}
Let $H$ be a  divisible group and $f\colon H\to G$ be a group homomorphism. If $G$ does not include $\mathbb{Q}$ and the Pr\"ufer $p$-group $\mathbb{Z}(p^\infty)$ 
for any prime $p$ as a subgroup, then $f(H)$ is trivial.
\end{corollary}
The very clever idea of the following proof was suggested to us by the anonymous referee.
\begin{proof} 
Supposing to the contrary that $f(H)$ is non-trivial we select $g_0\in f(H)-\{1\}$. Of course, $f(H)$ is divisible since $H$ is divisible, so assuming we have selected $g_n\in f(H)$ we select $g_{n+1}\in f(H)$ such that $g_{n+1}^{(n+1)!}=g_n$. Now the subgroup $\langle g_0,g_1,\dots\rangle\subseteq f(H)$ is a nontrivial abelian divisible group, so by Theorem \ref{structureTheorem} this subgroup must include a copy of $\mathbb{Q}$ or some $\mathbb{Z}(p^\infty)$, contradicting the assumption that $G$, and therefore $f(H)$, does not include such subgroups. 

\end{proof}

\section{Virtually lcH-slender groups}

A discrete group $G$ is called \textit{locally compact Hausdorff slender} (abbrev. \textit{lcH-slender}) if  any group homomorphism from a locally compact group to $G$ is continuous. This concept was first introduced by Conner and Corson in \cite{ConnerCorson}.
 
There are some obvious ``poison subgroups" that prevent a group $G$ from being \textit{lcH}-slender. These subgroups are $\mathbb{Q}$, the $p$-adic integers $\mathbb{Z}_p$, the Pr\"ufer $p$-group and torsion groups in general for the following reason:

There exist discontinuous homomorphisms from $\mathbb{R}$ to $\mathbb{Q}$. For example take a Hamelbasis $B$ of $\mathbb{R}=\bigoplus_{b\in B}b\mathbb{Q}$, then the map that maps any linear combination $\sum_{b\in B} bq_b$ with coefficients $q_b\in\mathbb{Q}$ to the sum of the coefficients in $\mathbb{Q}$ is discontinuous regarding the standard topology on $\mathbb{R}$. Similarly, one can construct a discontinuous homomorphism from the compact circle group $\mathbb{R}/\mathbb{Z}$ to the Pr\"ufer $p$-group $\mathbb{Z}(p^\infty)$ for any prime $p$.

If $G$ contains torsion, then $G$ includes $\mathbb{Z}/p\mathbb{Z}$ as a subgroup for some prime $p$. One can construct a discontinuous homomorphism from the compact group $\prod_\mathbb{N} \mathbb{Z}/p\mathbb{Z}$ by extending the homomorphism from $\bigoplus_\mathbb{N} \mathbb{Z}/p\mathbb{Z}$ to $\mathbb{Z}/p\mathbb{Z}$, which just takes the sum of all entries, via a vector space argument.

The $p$-adic integers $\mathbb{Z}_p$ form a non-discrete compact group for any prime $p$. Therefore, if $G$ includes $\mathbb{Z}_p$ for some $p$, then the identity homomorphism from the non-discrete compact $\mathbb{Z}_p$ to the discrete subgroup of $G$ isomorphic to $\mathbb{Z}_p$ is discontinuous.

The algebraic structure of $lcH$-slender groups was characterized by Corson and the third author in \cite[Thm. 1]{CorsonVarghese}. 

\begin{theorem}
\label{SamOlga}
A group $G$ is $lcH$-slender if and only if $G$ is torsion-free and does not include $\mathbb{Q}$ or any $p$-adic integer group $\mathbb{Z}_p$ as a subgroup.
\end{theorem}

We extend this concept to groups with torsion by studying \textit{virtually lcH-slender groups}, i. e. groups that contain a finite index $lcH$-slender subgroup. 

\begin{corollaryA}
A group $G$ is virtually $lcH$-slender if and only if $G$ is virtually torsion-free and does not include $\mathbb{Q}$ or the p-adic integers $\mathbb{Z}_p$ for any prime $p$ as a subgroup.
\end{corollaryA}

\begin{proof}
Suppose $G$ is virtually torsion-free and does not include $\mathbb{Q}$ or any $p$-adic integer group $\mathbb{Z}_p$ as a subgroup and let $H$ be a finite index torsion-free subgroup of $G$.  Then $H$ does not include $\mathbb{Q}$ or any $p$-adic integer group $\mathbb{Z}_p$ as a subgroup, so $H$ is $lcH$-slender by \cite[Thm. 1]{CorsonVarghese} and therefore $G$ is virtually $lcH$-slender.

Suppose now that $G$ is virtually $lcH$-slender. Then $G$ is virtually torsion-free. Let $H$ be a finite index $lcH$-slender subgroup of $G$. Without loss of the generality we can assume that $H$ is normal in $G$. Suppose that $G$ includes a subgroup $P$ which is either isomorphic to $\mathbb{Q}$ or to the group of $p$-adic integers $\mathbb{Z}_p$. 

Since $H$ has finite index in $G$, $P/(P\cap H)\cong PH/H$ is a finite quotient group of $P$.\\
\textit{Case 1:} $P\cong \mathbb{Q}$. Every quotient of $\mathbb{Q}$ is divisible, thus $\mathbb{Q}$ does not have a finite non-trivial quotient group, so $P/(P\cap H)$ must be trivial and therefore $P\subset H$, which is a contradiction to the assumption, that $H$ is \textit{lcH}-slender by \cite[Thm. 1]{CorsonVarghese}.\\
\textit{Case 2:} $P\cong \mathbb{Z}_p$ for some prime $p$. Since $P/P\cap H$ is finite, $P\cap H$ must be isomorphic to $\mathbb{Z}_p$ \cite[p. 18]{Fuchs}, which is a contradiction to the assumption, that $H$ is $lcH$-slender by Theorem \ref{SamOlga}.
\end{proof}

The class consisting of virtually torsion-free groups includes the following groups:
\begin{itemize}
\item (Outer) automorphism groups of right-angled Artin groups \cite[Thm. 5.2]{CharneyVogtmann}, Lemma \ref{TorsionNormal}, since ${\rm Inn}(A_\Gamma)$ is torsion-free. 
\item Finitely generated linear groups in characteristic $0$ \cite{Selberg}, in particular Coxeter groups \cite[Cor. 6.12.12]{Davis}.
\end{itemize}
	
In Proposition \ref{classG} we will show that all these groups contain neither $\mathbb{Q}$ nor $\mathbb{Z}_p$ for any prime $p$ as a subgroup, hence these groups are virtually $lcH$-slender.	
	
\section{Proof of Theorem B}

We recall some concepts from abelian group theory. An abelian group is \textit{algebraically compact} if it is an algebraic direct summand of a compact Hausdorff abelian group (see \cite[Section 6.1]{Fuchs}). In particular, a compact abelian group is algebraically compact, since it is trivially a direct summand of itself. An abelian group $A$ is \textit{cotorsion} if every short exact sequence of abelian groups,\\
\begin{center}
\begin{tikzcd}
0 \arrow[r] & A \arrow[r, "i"] & B \arrow[r, "\pi"] & C \arrow[r] & 0
\end{tikzcd}
\end{center}

with $C$ torsion-free, splits (that is, there is a homomorphism $\psi \colon C\to B$ such that $\pi\circ\psi$ is identity)\cite[Section 9.6]{Fuchs}. Algebraically compact groups are cotorsion and the homomorphic image of a cotorsion group is cotorsion (see \cite[p. 282]{Fuchs} for both assertions).\\
An abelian group is \textit{cotorsion-free} if $0$ is its only cotorsion subgroup \cite[p. 500]{Fuchs}. An abelian group is cotorsion-free if and only if it is torsion-free and does not include a copy of $\mathbb{Z}_p$ or $\mathbb{Q}$ \cite[Thm. 13.3.8]{Fuchs}.\\
A torsion abelian group $A$ is cotorsion if and only if it is of form $A=B\oplus D$ where $B$ is bounded (that is, there exists some positive natural number $n$ with $nB=0$) and $D$ is divisible \cite[Cor. 9.8.4]{Fuchs}. In particular, a finite abelian torsion group is cotorsion.

Recall, a group $H$ is called \textit{artinian}, if  every non-empty collection of subgroups of $H$ has a minimal element under subset partial order.

\begin{lemma}
\label{AbelianArtinian}
Let $A$ be an abelian group. The group $A$ is artinian if and only if $A$ can be written as a direct sum $A=J\oplus\bigoplus_{i=0}^m \mathbb{Z}(p_i^\infty)$ of a finite abelian group $J$ and finitely many Pr\"ufer $p_i$-groups.
\end{lemma}

\begin{proof}
By \cite[Thm. 4.5.3]{Fuchs} $A$ is artinian if and only if $A$ is a direct sum of a finite number of so called cocyclic groups, which are each isomorphic to either $\mathbb{Z}/p^k$ or $\mathbb{Z}(p^\infty)$ for some prime $p$ and some positive integer $k$ by \cite[Thm. 1.3.3]{Fuchs}.
\end{proof}

For the proof of Theorem B we need one more additional lemma. 
\begin{lemma}
\label{CompactTorsion}
Let $\varphi\colon K\to G$ be a group homomorphism from a compact group $K$ into a discrete group $G$.  Assume that $G$ does not include $\mathbb{Q}$, or the $p$-adic integers $\mathbb{Z}_p$ for any prime $p$,  or the Pr\"ufer $p$-group $\mathbb{Z}(p^\infty)$ for any prime $p$ as a subgroup. Also suppose that abelian torsion subgroups of $G$ are artinian. Then $\varphi(K)$ is a torsion group.
\end{lemma}
\begin{proof}
For $k\in K$ the subgroup $\overline{\langle k\rangle}$ of $K$ is a compact abelian group \cite[Lemma 4.4]{Stroppel}, therefore algebraically compact, therefore cotorsion. Thus the homomorphic image $A=\varphi(\overline{\langle k\rangle})$ is cotorsion. Let $\rm{Tor}(A)$ denote the torsion subgroup of $A$. As $A$ is an abelian subgroup of $G$ we know that $\rm{Tor}(A)$ is artinian, and therefore $\rm{Tor}(A)$ is a direct sum of a finite abelian group and some Pr\"ufer groups by Lemma \ref{AbelianArtinian}. But $G$ does not include any Pr\"ufer groups, so $\rm{Tor}(A)$ is finite.\\
Now, $\rm{Tor}(A)$ is a finite abelian group, and therefore cotorsion. Since $A/\rm{Tor}(A)$ is torsion-free, the short exact sequence\\
\begin{center}
\begin{tikzcd}
0 \arrow[r] & \rm{Tor}(A) \arrow[r] & A \arrow[r] & A/\rm{Tor}(A) \arrow[r] & 0
\end{tikzcd}
\end{center}

splits, and so we may write $A\cong \rm{Tor}(A)\oplus A/\rm{Tor}(A)$. Now $A/\rm{Tor}(A)$ is torsion-free and it also does not include $\mathbb{Q}$ or any $\mathbb{Z}_p$ since it is isomorphic to a subgroup of $A$ (hence isomorphic to a subgroup of $G$). Therefore $A/\rm{Tor}(A)$ is cotorsion-free. But $A/\rm{Tor}(A)$ is a homomorphic image of the cotorsion group $A$, and therefore is itself cotorsion. Thus $A/\rm{Tor}(A)$ is trivial. In particular $\rm{Tor}(A)=A$ and $A$ is finite. More particularly $\varphi(k)$ has finite order. Hence $\varphi(K)$ is a torsion group.
\end{proof}

\begin{proof}[{\bf Proof of Theorem B}]
We will first show the part of the theorem in which the target group only satisfies conditions (i)-(iii). The statement for a group also satisfying the fourth condition will be an easy corollary of that case.
	
Let $\varphi\colon L\to G$ be a group homomorphism from a locally compact group $L$ into a group $G$ that satisfies the conditions (i)-(iii) of Theorem B. Without loss of generality we can assume $\varphi(L)=G$, since every subgroup of $G$ satisfies (i)-(iii) of Theorem B. 
	
\textit{Step 1: $\varphi(L^\circ)$ is trivial.}\\
Due to Iwasawa's structure Theorem \ref{IwasawaVanDantzig} we can write the connected component $L^\circ$ as $L^\circ=H_1\cdot...\cdot H_k\cdot K$, where each $H_i$ is isomorphic to $\mathbb{R}$ and $K$ is a compact connected group. Due to Lemma \ref{Rtrivial}, the image $\varphi(H_i)$ of each $H_i$ is trivial. Also, since $K$ is compact connected it is divisible \cite[Cor. 2]{Mycielski}, and therefore $\varphi(K)$ is also trivial by Lemma \ref{CompactTorsion}.  

Since $\varphi(L^\circ)$ is trivial, we know that $\varphi$ descends to a homomorphism $\overline{\varphi}:L/L^\circ\to G$:
\begin{center}
\begin{tikzcd}
L \arrow[rr, "\varphi"] \arrow[swap, rd, "\pi"] &                                       & G \\
                                          & L/L^\circ \arrow[swap, ru, "\overline{\varphi}"] &  
\end{tikzcd}

\end{center}

\textit{Step 2: Apply van Dantzig's Theorem to $L/L^\circ$:}\\
Since $L/L^\circ$ is a totally disconnected locally compact group, we apply van Dantzig's Theorem \ref{IwasawaVanDantzig} to find at least one compact open subgroup $K_1\subseteq L/L^\circ$. This allows us to differentiate the following two cases:\\
\textit{Case A:} 
There exists a compact open subgroup $K\subseteq L/L^\circ$ such that $\overline{\varphi}(K)$ is trivial. Then $\pi^{-1}(K)$ is open in $L$ and completely inside the kernel of $\varphi$, so $\ker(\varphi)$ is open. Hence the map $\varphi$ is continuous.\\
\textit{Case B:} There is no such compact open subgroup. This case requires more separate steps:

\textit{Step B.1: Find a ``minimal" compact open subgroup $K_0\subseteq L/L^\circ$ using condition (iii):}\\
By Lemma \ref{CompactTorsion} we know that for every compact open subgroup $K$, the image $\overline{\varphi}(K)$  is a torsion group. We consider the following family of torsion subgroups of $G$:
$$\mathcal{T}=\{\bar{\varphi}(K)\big| K\subseteq L/L^\circ\text{ compact open subgroup}\}.$$
Since torsion subgroups of $G$ are artinian, the set  $\mathcal{T}$ has a minimal element, we choose one and we call it $\overline{\varphi}(K_0)$.  

\textit{Step B.2: We now show that $\overline{\varphi}(K_0)$ is a normal subgroup, hence the theorem holds for t.d.l.c. groups:}\\
The group $gK_0g^{-1}$ is a compact open subgroup of $L/L^\circ$ for all $g\in L/L^\circ$, so $gK_0g^{-1} \cap K_0$ is as well. Thus we have $\bar{\varphi}(gK_0g^{-1}\cap K_0)\subseteq \bar{\varphi}(K_0)$ and due to minimality we have $\bar{\varphi}(gK_0g^{-1}\cap K_0)= \bar{\varphi}(K_0)$. That means $\bar{\varphi}(gK_0g^{-1})=\bar{\varphi}(K_0)$ for every $g\in L/L^\circ$ and therefore $\bar{\varphi}(L/L^\circ)$ is contained in the normalizer ${\rm Nor}_{L/L^\circ}(\bar{\varphi}(K_0))$. By assumption the map $\varphi$ is surjective, therefore $G=\overline{\varphi}(L/L^\circ)={\rm Nor}(\bar{\varphi}(K_0))$. We now set $N_0:=\bar{\varphi}(K_0)$.

\textit{Step B.3: Get the desired result for $\varphi$.}\\
The subgroup  $\overline{\varphi}^{-1}(\overline{\varphi}(K_0))$ is an open normal subgroup in $L/L^\circ$. Therefore $N:=\pi^{-1}(\overline{\varphi}^{-1}(\overline{\varphi}(K_0))$ is an open normal subgroup of $L$, due to the continuity of $\pi$, and $\varphi(N)=\overline{\varphi}(K_0)$ which is torsion. 
This concludes the proof if $G$ satisfies conditions (i)-(iii).

If $G$ additionally satisfies condition (iv), then $\varphi(N)$ has to be trivial. But then $N\subseteq \ker(\varphi)$, so the kernel is open and $\varphi$ is continuous.
\end{proof}

If $L$ is an \textit{almost connected} locally compact group (that is, a locally compact group where $L/L^\circ$ is compact) we might strengthen Theorem B:

\begin{corollary}
If $G$ satisfies conditions (i) and (ii) of Theorem B and \textbf{abelian} torsion subgroups are artinian, then for any homomorphism $\varphi\colon L\to G$ from an almost connected locally compact group $L$ to $G$ the image $\varphi(L)$ is a torsion group.
\end{corollary}
\begin{proof}

Analogous to step 1 of the proof of Theorem B we see, that $\varphi(L^\circ)$ is trivial. Since $L/L^\circ$ is compact the image $\bar{\varphi}(L/L^\circ)$ is torsion by Lemma \ref{CompactTorsion}. But then $\varphi(L)=\bar{\varphi}(L/L^\circ)$ is also torsion.
\end{proof}

In particular, every homorphism from an almost connected locally compact group into a CAT$(0)$ group has torsion image (for the proof see the proof of Proposition 5.2 (3)).

\section{On discrete groups with properties (i)-(iii) of Theorem B}
In this last section we collect many groups that arise in geometric group theory and satisfy the conditions (i)-(iii) of Theorem B. 
\medskip

Let $\mathcal{G}$ denote the class of all groups $G$ with the following three properties:
\begin{enumerate}
\item[(i)] $G$ does not include $\mathbb{Q}$ or the $p$-adic integers $\mathbb{Z}_p$ 	for any prime $p$ as a subgroup,
\item[(ii)] $G$ does not include the Pr\"ufer $p$-group $\mathbb{Z}(p^\infty)$ for any prime $p$ as a subgroup and
\item[(iii)]  torsion subgroups in $G$ are artinian.
\end{enumerate}
The class $\mathcal{G}$ is closed under taking subgroups. Furthermore, this class is closed under taking finite graph products of groups and group extensions, which we will show in Section 5.3. 

An algebraic property of a group that prohibits the existence of subgroups isomorphic to $\mathbb{Q}$ or $\mathbb{Z}(p^\infty)$ is residual finiteness. We recall that a group $G$ is said to be \textit{residually finite} if for any $g\in G-\left\{1_G\right\}$ there exists a finite group $F_g$ and a group homomorphism $f\colon G\to F_g$ such that $f(g)\neq 1$. It follows from the definition that any subgroup of a residually finite group is also residually finite. 
\begin{lemma}
\label{ResFinite}
Let $G$ be a group. If $G$ is residually finite, then $G$ does not have non-trivial divisible subgroups. In particular, a residually finite group does not contain $\mathbb{Q}$ or the Pr\"ufer $p$-group $\mathbb{Z}(p^\infty)$ for any prime $p$ as a subgroup.
\end{lemma}
\begin{proof}
Assume that $G$ has a non-trivial divisible subgroup $D$. For $g\in D-\left\{1_G\right\}$ there exists a finite group $F$ and a group homomorphism $f\colon D\to F$ such that $f(g)\neq 1_F$. Since $D$ is divisible, there exists $d\in D$ such that $d^{ord(F)}=g$, where ${\rm ord}(F)$ is the cardinality of the group $F$. We obtain 
$$f(g)=f(d^{{\rm ord}(F)})=f(d)^{{\rm ord}(F)}=1_F.$$ 
This contradicts the fact that $f(g)\neq 1_F$, thus $D$ is trivial. 
\end{proof}

Our goal is to convince the reader that the class $\mathcal{G}$ is huge and contains  many geometric groups, i.e. groups with nice actions on metrically injective spaces or Gromov-hyperbolic spaces or ${\rm CAT}(0)$ spaces. For a definition of a metrically injective group see Subsection 5.1 and for definitions and further properties of Gromov-hyperbolic groups and ${\rm CAT}(0)$ groups we refer to \cite{BridsonHaefliger}.

\begin{proposition}
\label{classG}
If $G$ is
	\begin{enumerate} 
	 \item a virtually $lcH$-slender group,
	 \item a cocompactly cubulated ${\rm CAT}(0)$ group,
	\item a ${\rm CAT}(0)$ group whose torsion subgroups are artinian, e.g. a Coxeter group,
	\item a metrically injective group whose torsion subgroups are artinian, e.g. a Helly group whose torsion subgroups are artinian,
	\item a Gromov-hyperbolic group,
	\item  a finitely generated residually finite group whose torsion subgroups are artinian, e.g. the (outer) automorphism group of a right-angled Artin group,
	\item a one-relator group,
	\item a finitely generated linear group in characteristic $0$,
	\item the Higman group, 	
	\end{enumerate}
then $G$ is in the class $\mathcal{G}$.
\end{proposition}
\begin{proof}
First of all we note that the groups in (2)-(9) are all finitely generated, in particular countable. Thus these groups can not have $\mathbb{Z}_p$ as a subgroup since this group is uncountable.

\textit{to (1):} Let $G$ be a virtually $lcH$-slender group. By Corollary A the group $G$ does not include $\mathbb{Q}$ or  $\mathbb{Z}_p$ for any prime $p$ as a subgroup. Further, the group $G$ is virtually torsion-free, in particular $G$ has a normal torsion-free subgroup $H$ of finite index. Let $T\subseteq G$ be a torsion subgroup. We consider the map $\pi\circ\iota \colon T\hookrightarrow G\twoheadrightarrow G/H$. Since $H$ is torsion-free, this map $\pi\circ\iota$ is injective and therefore $T$ has to be finite. Thus, the group $G$ is in the class  $\mathcal{G}$.

\textit{to (2) and (3):} It is known that an abelian subgroup of a ${\rm CAT}(0)$ group is finitely generated, see \cite[Thm. I.4.1]{Davis}, thus a ${\rm CAT}(0)$ group does not have $\mathbb{Q}$ as a subgroup. Further, a {\rm CAT}(0) group has only finitely many conjugacy classes of finite subgroups \cite[Thm. I.4.1]{Davis}, hence there is a bound on the order of finite order elements in a ${ \rm CAT}(0)$ group and therefore this group can not have $\mathbb{Z}(p^\infty)$ as a subgroup. Furthermore, it is known that torsion subgroups in Coxeter groups are finite \cite[Prop. 7.4]{KramerVarghese}. Finiteness of torsion subgroups of cocompactly cubulated groups follows from \cite[Cor. G]{CapraceSageev}.\\
Despite the fact that ${\rm CAT}(0)$ groups are very popular, the question regarding the finiteness of torsion subgroups of general ${\rm CAT}(0)$ groups is still open. 

\textit{to (4):} See Section 5.1, in particular Lemma \ref{divisible}, Lemma \ref{torsion} and Corollary  \ref{Helly}. 

\textit{to (5):} Let $G$ be a Gromov-hyperbolic group. It is known that torsion subgroups of Gromov-hyperbolic groups  are finite \cite[Chap. 8 Cor. 36]{Ghys}, thus $\mathbb{Z}(p^\infty)$ is not a subgroup of $G$ and any torsion subgroup of $G$ is artinian. Further, it was proven in \cite[Cor. 6.8]{Helly} that any Gromov-hyperbolic group is a Helly group. Thus, by (4) the group $G$ is in the class $\mathcal{G}$.

\textit{to (6):}  By Lemma \ref{ResFinite} we know that a residually finite group does not include $\mathbb{Q}$ or $\mathbb{Z}(p^\infty)$ as a subgroup. Hence, a finitely generated residually finite group whose torsion subgroups are artinian is in the class $\mathcal{G}$. 

Given a right-angled Artin group $A_\Gamma$, there exist a number $m\in\mathbb{N}$ and an  injective group homomorphism $A_\Gamma\hookrightarrow{ \rm GL}_m(\mathbb{R})$
\cite[Cor. 3.6]{HsuWise}, thus any right-angled Artin group is a finitely generated linear group and as such it is residually finite \cite{Malcev}. Furthermore, it was proven in \cite{Baumslag} that the automorphism group of a finitely generated residually finite group is also residually finite, thus ${\rm Aut}(A_\Gamma)$ is residually finite. It was proven in \cite[Thm. 4.2]{CVout} that also the outer automorphism group of a right-angled Artin group is residually finite. Hence, by Lemma \ref{ResFinite} the (outer) automorphism group of a right-angled Artin group does not include $\mathbb{Q}$ or $\mathbb{Z}(p^\infty)$ as a subgroup.  Torsion subgroups in the (outer) automorphism group of a right-angled Artin group are finite, which follows from the fact that this group is virtually torsion-free, as we already mentioned before.

\textit{to (7):} Given a one-relator group $G$, there are two possibilities: (i) $G$ has torsion elements, (ii) $G$ is torsion-free. In the first case it was proven in \cite[Thm. IV.5.5]{Lyndon} that $G$ is hyperbolic, thus by (3) this group is in the class $\mathcal{G}$. If $G$ is torsion-free, then $G$ does not include $\mathbb{Q}$ or $\mathbb{Z}_p$ as a subgroup \cite[Thm. 1]{Newman}.

\textit{to (8):} Let $G$ be a finitely generated linear group in characteristic $0$. It was proven in \cite{Selberg} that $G$ is virtually torsion-free, hence torsion subgroups of $G$ are always finite. Further, it is known that $G$ is residually finite \cite{Malcev}, thus by (5) this group is in the class $\mathcal{G}$.

\textit{to (9):} The Higman group was defined in \cite{Higman} and is an amalgam of two torsion-free groups \cite[Thm. IV.2.7]{Lyndon}, thus it is torsion-free by \cite[\S 6.2 Prop. 21]{Serre}. 
Further, this group does not include $\mathbb{Q}$ as a subgroup, which follows from
\cite[Thm. E]{Martin} .

\end{proof}

\subsection{Geometric groups}
To prove the fact that a metrically injective group $G$ is in the class $\mathcal{G}$, we first need to gather some information about these groups. The goal is to briefly introduce the framework in which these groups are studied, state the results we need in order to prove Proposition \ref{classG} (4) and give its proof.

Let $(X,d)$ be a metric space. For an isometry $f\colon X\to X$ the \textit{translation length of }$f$, denoted by $|f|$ is defined as $|f|:={\rm inf}\left\{d(x,f(x))\mid x\in X\right\}$. We have the following 
general classification of the isometries of a metric space $X$ in terms of ${\rm Min}(f):=\left\{x\in X\mid d(x, f(x))=|f|\right\}$: An isometry $f$ is said to be \textit{parabolic} if ${\rm Min}(f)=\emptyset$ and \textit{semi-simple} otherwise. In the latter case, $f$ is \textit{elliptic} if 
for any $x\in X$ the subset $\left\{f^n(x)\mid n\in \mathbb{N}\right\}\subseteq X$  is bounded and \textit{hyperbolic} if it is unbounded. 

For a group $G$ and a metric space $X$, a group action $\Phi\colon G\to{\rm Isom}(X)$ is called \textit{proper} if for each $x\in X$ there exists a real number $r>0$ such that the set  $\left\{g\in G\mid \Phi(g)(B_r(x))\cap B_r(x)\neq\emptyset\right\}$ is finite and $\Phi$ is called \textit{cocompact} if there exists a compact subset $K\subseteq X$ such that $\Phi(G)(K)=X$. In geometric group theory it is common to work with actions that have both properties, such an action is called \textit{geometric}. 

The main idea of  geometric group theory is, given a group $G$, to construct a geometric action on a metric space $X$ with rich geometry and to use this geometry to prove algebraic results about the group $G$. Here, the spaces will be metrically injective spaces, while the algebraic property of the groups will be the occurrence of subgroups isomorphic to $\mathbb{Q}$, the $p$-adic integers $\mathbb{Z}_p$ or the Pr\"ufer $p$-group $\mathbb{Z}(p^\infty)$. 

Before proving Proposition \ref{classG} (4) we need to discuss a few useful results. We first need to study groups that are \textit{almost divisible}, that is, 
there is an unbounded sequence $(n_i)_{i\in \mathbb{N}}$ of natural numbers such that the group is $n_i$ divisible for each $n_i$. We note that the groups $\mathbb{Q}$, the $p$-adic integers $\mathbb{Z}_p$ and the Pr\"ufer $p$-group $\mathbb{Z}(p^\infty)$ are examples of almost divisible groups.

\begin{proposition}
	\label{toolQ}
	Let $\Phi\colon G\to{\rm Isom}(X)$ be a geometric action on a metric space $X$. 	
If for any hyperbolic isometry $\Phi(g)\in \Phi(G)$  we have $|\Phi(g)^n|\geq n\cdot |\Phi(g)|$ for all $n\in\mathbb{N}$, then any almost divisible subgroup $H\subseteq G$ is a torsion group.
\end{proposition}
\begin{proof}
First of all we note  that every isometry in $\Phi(G)$ is semi-simple \cite[Thm. II.6.2.10]{BridsonHaefliger}.  Further, it was shown in \cite[Thm. II.6.2.10]{BridsonHaefliger} that if $G$ acts geometrically on a metric space $X$, then the infimum of translation length of hyperbolic isometries is always positive. 

Let $H\subseteq G$ be an almost divisible subgroup. Our goal is to prove that for $h\in H$ the order of $h$ has to be finite. First, we can see that if ${\rm ord}(\Phi(h))< \infty$, then ${\rm ord}(h)<\infty$ since the action is proper, so in that case $h$ is a torsion element. If ${\rm ord}(\Phi(h))=\infty$, then $|\Phi(h)|>0$, because else the action would not be proper. Now $H$ is almost divisible, so there exists an unbounded sequence $(n_i)_{i\in \mathbb{N}}$ such that there exist elements $d_i\in H$ with $h=d_i^{n_i}$. By assumption we  have 
$$|\Phi(h)|=|\Phi(d_i^{n_i})|=|\Phi(d_i)^{n_i}|\geq n_i\cdot |\Phi(d_i)|.$$ 
But since the sequence $(n_i)_{i\in\mathbb{N}}$ is unbounded, that means that the sequence of translation lengths $(|\Phi(d_i)|)_{i\in \mathbb{N}}$ converges to zero, which in turns contradicts the fact that the infimum of translation lengths of hyperbolic isometries is positive.
\end{proof}

The result of Proposition \ref{toolQ} gives us a tool to prove that a group which we are interested in does not include $\mathbb{Q}$ or the $p$-adic integers $\mathbb{Z}_p$ as a subgroup.

\begin{remark}\label{finEmb}
	The inequality condition on the translation lengths of hyperbolic isometries is necessary, since any finitely generated group $G$ acts geometrically on its Cayley-graph ${\rm Cay}(G,S)$ where $S$ is a finite generating set of $G$. There are finitely generated groups that contain divisible torsion-free subgroups, e.g. $\mathbb{Q}$, as can be found in  \cite[Thm. IV]{Neumann}, there even exist finitely presented groups that contain $\mathbb{Q}$ \cite[Thm 1.4, Prop. 1.10]{finitelypresented}.
\end{remark}

Now we are interested in actions on metric spaces, which satisfy the condition on hyperbolic isometries from Proposition \ref{toolQ}. One possibility for this is to study actions on injective metric spaces.

We say a metric space $(X,d)$ is \textit{injective}, if it is an injective object in the category of metric spaces and $1$-Lipschitz maps. 
Examples of injective metric spaces are $\mathbb{R}$-trees or more generally finite dimensional CAT$(0)$ cube complexes with the $l^{\infty}$ metric (see \cite{Bowditch} for more details). Following \cite{Haettel}, we call a group $G$ \textit{metrically injective} if it acts geometrically on an injective metric space.  

A very important tool for studying these spaces are \textit{bicombings}. Given a geodesic metric space $(X,d)$, we call a map $\sigma\colon X\times X\times [0,1]\to X$ a \textit{bicombing} if the family of maps $\sigma_{xy}:=\sigma(x, y, \cdot)\colon [0,1]\to X $ satisfies the following three properties:
\begin{itemize}
\item $\sigma_{xy}$ is a constant speed geodesic from $x$ to $y$, that is $\sigma_{xy}(0)=x, \sigma_{xy}(1)=y$ and $d(\sigma_{xy}(s), \sigma_{xy}(t))=\mid t-s\mid d(x,y)$ for $s, t\in [0,1]$ and $x,y\in X$. 
\item $\sigma_{yx}(t)=\sigma_{xy}(1-t)$ for $t\in [0,1]$ and $x,y\in X$.
\item $d(\sigma_{xy}(t), \sigma_{vw}(t))\leq (1-t)d(x,v)+td(y,w)$ for $t\in [0,1]$ and $x,y,v,w\in X$.
\end{itemize}
Given an isometry $\gamma\in {\rm Isom}(X)$, we say the bicombing $\sigma$ is $\gamma$\textit{-equivariant} if $\gamma\left(\sigma(x,y,t)\right)=\sigma(\gamma(x),\gamma(y),t)$ for all $x,y\in X$ and $t\in [0,1]$.
It was proven in \cite[Prop. 3.8]{Lang} that an injective metric space $X$ always admits an ${\rm Isom}(X)$-equivariant bicombing.

We now want to construct ``barycenter" maps for injective metric spaces to prove the inequality for translation lengths of hyperbolic isometries of the previous proposition holds in injective metric spaces. Descombes and Lang observed in \cite{bicombings} that the barycenter maps given in \cite{Navas} for Busemann spaces  translate to metrically injective spaces. Here we include the complete construction of these maps.
\begin{lemma}
\label{bar}
Let $X$ be an injective metric space. 
\begin{enumerate}
	\item For every $n\in\mathbb{N}$ there exists a map ${\rm bar}_n\colon X^n\to X$ such that for all $x_1,...,x_n,y_1,...,y_n\in X$:
		\begin{enumerate}
		\item $d({\rm bar}_n(x_1, x_2, \ldots, x_n), {\rm bar}_n(y_1, y_2, \ldots, y_n))\leq\frac{1}{n}\Sigma^n_{i=1} d(x_i, y_i)$,
		\item ${\rm bar}_n(x_1, x_2, \ldots, x_n)={\rm bar}_n(x_{\pi(1)}, x_{\pi(2)}, \ldots, x_{\pi(n)})$ for every $\pi\in{\rm Sym}(n)$ and
		\item $\gamma({\rm bar}_n(x_1, x_2, \ldots, x_n))={\rm bar}_n(\gamma(x_1), \gamma(x_2),\ldots, \gamma(x_n))$ for every isometry $\gamma\in{\rm Isom}(X)$.
		\end{enumerate}
	\item Let $\varphi\in {\rm Isom}(X)$ be an isometry. For every $n\in\mathbb{N}$  and $x\in X$, the inequality $|\varphi| \leq\frac{1}{n} d(x, \varphi^n(x))$ holds.
	\end{enumerate}		
\end{lemma}
\begin{proof}
Let $X$ be an injective metric space and $\sigma$ be an ${\rm Isom}(X)$-equivariant bicombing. 
We first define ${\rm bar}_1\colon X\to X$ as ${\rm bar}_1(x):=x$ for $x\in X$ and ${\rm bar}_2:X\times X\to X$ as 
${\rm bar}_2(x,y):=\sigma_{xy}(\frac{1}{2})$ for $x,y\in X$. The maps ${\rm bar}_1$ and ${\rm bar}_2$ obviously satisfy conditions (a), (b) and (c).
Now assuming that ${\rm bar}_n$ has been defined and satisfies the properties (a), (b) and (c) we define ${\rm bar}_{n+1}((x_1,....,x_{n+1}))$ as follows:

For a tupel $z=(z_1,z_2,...,z_n, z_{n+1})\in X^{n+1}$ we set 
$\hat{z}^{(i)}:=(z_1,z_2,...,z_{i-1},z_{i+1},...,z_n, z_{n+1}).$ 
We now define a sequence 
$(y_k)_{k\in\mathbb{N}}\text{ where each }y_k=(y_{1k}, ...., y_{(n+1)k})\in X^{n+1}$. We start with 
$y_1:=(x_1, ...,x_{n+1}).$ 
For $k\geq 2$,  $y_{k}$ is defined by recursively applying ${\rm bar}_n$ to $\hat{y}^{(i)}_{k-1}$. More precisely: 
$$y_{ik}:={\rm bar}_n(\hat{y}^{(i)}_{k-1})\text{ for }i\in\{1,...,n+1\}.$$
One can show, that each sequence $\left(y_{ik}\right)_{k\in \mathbb{N}}$ for $i=1, ..., n+1$ is a Cauchy-sequence and therefore convergent since $X$ is complete (see \cite[Section 1]{Navas}). Additionally, one can show that $\lim_{k\to \infty}y_{ik}=\lim_{k\to \infty}y_{jk}$ for all $i,j\in \{1,...,n+1\}$. Therefore we can define ${\rm bar}_{n+1}(x_1,...,x_{n+1}):=\lim_{k\to \infty}y_{ik}$ for some $i\in\{1,...,n+1\}$ and the limit is independent of the choice of $i$, this can also be found in \cite[Section 1]{Navas}. 

The fact that properties (b) and (c) are satisfied is immediate from the construction, checking property (a) is a lot more work and is done carefully in \cite[Section 1]{Navas}.

As an example, we can visualize the construction of bar$_4(x)$ with $x=(x_1,x_2,x_3,x_4)$, $x_i\in(\mathbb{R}^2,\ell^\infty)$ as follows.
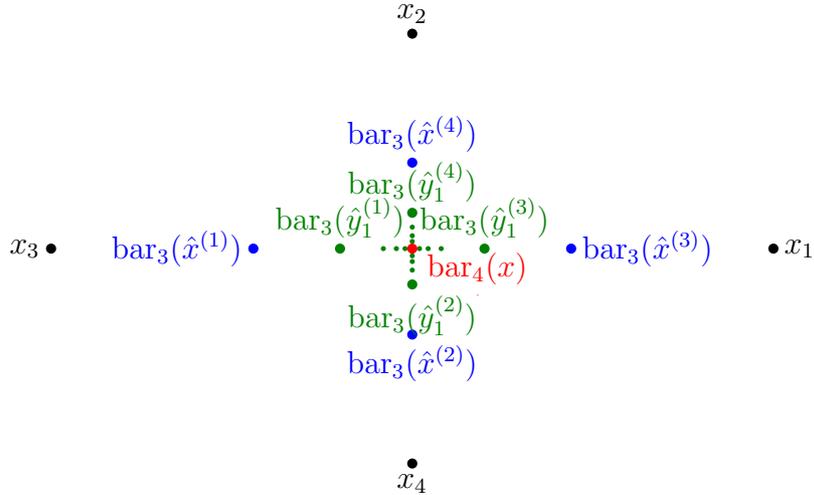
\begin{center}
	\begin{tikzpicture}[scale=0.95]
		\coordinate[label=right: {$x_1$}] (A) at (5,0);
		\coordinate[label=above: {$x_2$}] (B) at (0,3);
		\coordinate[label=left: {$x_3$}] (C) at (-5,0);
		\coordinate[label=below: {$x_4$}] (D) at (0,-3);

		\fill (A) circle (2pt);
		\fill (B) circle (2pt);
		\fill (C) circle (2pt);
		\fill (D) circle (2pt);

		\fill[color=blue] (-2.2,0) circle (2pt);
		\draw[color=blue] (-2.2,0) -- (-2.2,0) node[left]{${\rm bar}_3(\hat{x}^{(1)})$};
		
		\fill[color=blue] (2.2,0) circle (2pt);
		\draw[color=blue] (2.2,0) -- (2.2,0) node[right]{${\rm bar}_3(\hat{x}^{(3)})$};
		
		\fill[color=blue] (0,1.2) circle (2pt);
		\draw[color=blue] (0,1.2) -- (0,1.2) node[above]{${\rm bar}_3(\hat{x}^{(4)})$};
		
		\fill[color=blue] (0,-1.2) circle (2pt);
		\draw[color=blue] (0,-1.2) -- (0,-1.2) node[below]{${\rm bar}_3(\hat{x}^{(2)})$};

		\fill[color=dgreen] (-1,0) circle (2pt);
		\draw[color=dgreen] (-1,0) -- (-1,0) node[above]{${\rm bar}_3(\hat{y}_1^{(1)})$};
		
		\fill[color=dgreen] (1,0) circle (2pt);
		\draw[color=dgreen] (1,0) -- (1,0) node[above]{${\rm bar}_3(\hat{y}_1^{(3)})$};
		
		\fill[color=dgreen] (0,0.5) circle (2pt);
		\draw[color=dgreen] (0,0.5) -- (0,0.5) node[above]{${\rm bar}_3(\hat{y}_1^{(4)})$};
		
		\fill[color=dgreen] (0,-0.5) circle (2pt);
		\draw[color=dgreen] (0,-0.5) -- (0,-0.5) node[below]{${\rm bar}_3(\hat{y}_1^{(2)})$};
		
		\fill[color=dgreen] (-0.1,0) circle (1pt);
		\fill[color=dgreen] (-0.22,0) circle (1pt);
		\fill[color=dgreen] (-0.4,0) circle (1pt);
		
		\fill[color=dgreen] (0.1,0) circle (1pt);
		\fill[color=dgreen] (0.22,0) circle (1pt);
		\fill[color=dgreen] (0.4,0) circle (1pt);
		
		\fill[color=dgreen] (0,0.1) circle (1pt);
		\fill[color=dgreen] (0,0.3) circle (1pt);
		\fill[color=dgreen] (0,0.18) circle (1pt);
		
		\fill[color=dgreen] (0,-0.1) circle (1pt);
		\fill[color=dgreen] (0,-0.3) circle (1pt);
		\fill[color=dgreen] (0,-0.18) circle (1pt);
		
		\fill[color=red] (0,0) circle (2pt);
		\draw[color=red] (0.9,-0.65)--(0.9,-0.65) node[above]{${\rm bar}_4(x)$};
	\end{tikzpicture}
\end{center}
For the second part let $\varphi\in{\rm Isom}(X)$ be an isometry and $n\in\mathbb{N}$. For $x\in X$ we have:
\begin{align*}
|\varphi| &\overset{\text{Def.}}{=}{\rm inf}\left\{d(y,\varphi(y)\mid y\in X\right\} \\
&\leq d({\rm bar}_n(x, \varphi(x), \varphi^2(x), \ldots, \varphi^{n-1}(x)), \varphi({\rm bar}_n(x, \varphi(x), \varphi^2(x), \ldots, \varphi^{n-1}(x)))) \\
&\overset{(1c)}{=} d( {\rm bar}_n(x, \varphi(x), \varphi^2(x), \ldots, \varphi^{n-1}(x))  ,{\rm bar}_n(\varphi(x), \varphi^2(x), \varphi^3(x), \ldots, \varphi^{n}(x))) \\
&\overset{(1b)}{=} d( {\rm bar}_n(x, \varphi(x), \varphi^2(x), \ldots, \varphi^{n-1}(x)), {\rm bar}_n(\varphi^n(x), \varphi(x), \varphi^{2}(x), \ldots, \varphi^{n-1}(x))) \\
&\overset{(1a)}{\leq} \frac{1}{n}\cdot( d(x, \varphi^n(x))+d(\varphi(x), \varphi(x))+\ldots+d(\varphi^{n-1}(x), \varphi^{n-1}(x))) \\
&=\frac{1}{n} d(x, \varphi^n(x))
\end{align*}
\end{proof}

We note that the statement of the following lemma follows from various results of \cite{bicombings}. For the sake of completeness we give the proof of it here following the ideas of \cite{bicombings}.

\begin{lemma}\label{divisible}
	Suppose $G$ is a metrically injective group and $H\leq G$ a subgroup. If 
	$H$ is almost divisible, then $H$ is a torsion group. In particular, a metrically injective group does not have $\mathbb{Q}$ nor $\mathbb{Z}_p$ as a subgroup.
\end{lemma}
\begin{proof}
	Let $\Phi\colon G\to {\rm Isom}(X)$ be a geometric action of $G$ on an injective metric space $X$. Let $\Phi(g)$ be a hyperbolic isometry and $n\in\mathbb{N}$. By Lemma \ref{bar}
for $x\in{\rm Min}(\Phi(g)^n)$ the inequality 
$$|\Phi(g)| \leq\frac{1}{n} d(x, \Phi(g)^n(x))=\frac{1}{n}|\Phi(g)^n|$$
holds, hence $n\cdot |\Phi(g)| \leq |\Phi(g)^n|$. It follows by Proposition \ref{toolQ} that any almost divisible subgroup of $G$ is a torsion group. Since both $\mathbb{Q}$ and the 
	$p$-adic integers $\mathbb{Z}_p$ are (almost) divisible and torsion-free, the group $G$ has neither $\mathbb{Q}$ nor $\mathbb{Z}_p$ as a subgroup.
\end{proof}

We want to point out that an injective metric space is contractible \cite[Chap. 2]{Bowditch} and it is geodesic \cite[\S 2]{Lang}. Thus, a metrically injective group is finitely presented due to \cite[Cor. I.8.11]{BridsonHaefliger}. Therefore it cannot contain $\mathbb{Z}_p$, since this group is uncountable. However, as we noted in Remark \ref{finEmb}, this is not enough on its own to prevent the existence of subgroups isomorphic to $\mathbb{Q}$.

\begin{lemma}\label{torsion}
The order of elements in torsion subgroups of a metrically injective group is bounded. In particular, a metrically injective group does not include the Pr\"ufer $p$-group $\mathbb{Z}(p^\infty)$ for any prime $p$ as a subgroup.
\end{lemma}
\begin{proof}
We show that $G$ has finitely many conjugacy classes of finite subgroups. 
We first use \cite[Prop. 1.2]{Lang} to see that any finite subgroup of $G$ fixes a point. Then we can use \cite[Prop I.8.5]{BridsonHaefliger} and obtain that there are only finitely many conjugacy classes of isotropy subgroups, these are all finite due to the properness of the action. Since every torsion element is mapped into such an isotropy subgroup, its order then needs to be bounded, since the kernel is also finite due to properness.

Since the order of elements in the Pr\"ufer $p$-group $\mathbb{Z}(p^\infty)$ is unbounded, a metrically injective group can not have this group as a subgroup.
\end{proof}

Helly graphs are discrete versions of injective metric spaces. More precisely,  a connected graph is called \textit{Helly} if any family of pairwise intersecting combinatorial balls has a non-empty global intersection. A group $G$ is called \textit{Helly} if it acts geometrically by simplicial isometries on a Helly graph.  For example all Gromov-hyperbolic groups are Helly as well as groups acting geometrically on a ${\rm CAT}(0)$ cube complex, see \cite[Prop. 6.1, Cor. 6.8]{Helly}.

 Since a Helly group is also metrically injective \cite[Thm. 1.5]{Helly}, we obtain
 \begin{corollary}
  \label{Helly}
Any Helly group does not include $\mathbb{Q}$ or $\mathbb{Z}_p$ or $\mathbb{Z}(p^ \infty)$ for any prime $p$ as a subgroup.
 \end{corollary}

\subsection{The class $\mathcal{G}$ and graph products}
In this subsection we show that the class $\mathcal{G}$ is closed under taking extensions and graph products. 

The following lemma shows that the property of a group to have only artinian torsion subgroups is inherited by taking extensions.
\begin{lemma}\label{artinian}
	Let $\left\{1\right\}\to A\overset{\iota}{\to} B\overset{\pi}{\to} C\to\left\{1\right\}$ be a short exact sequence of groups. If torsion subgroups of $A$ and $C$ are artinian, then torsion subgroups of $B$ are also artinian. In particular, if $S,T$ are two groups whose torsion subgroups are artinian, then the torsion subgroups of $S\times T$ are also artinian.
\end{lemma}
\begin{proof}
For a torsion subgroup $H$ in $B$ we have a short exact sequence of torsion groups
$$\left\{1\right\}\rightarrow H\cap \iota(A)\rightarrow H\rightarrow \pi(H)\rightarrow\left\{1\right\}.$$
By assumption the torsion groups $H\cap\iota(A)$ and $\pi(H)$ are artinian. Further, it is known that being artinian is preserved by taking extensions \cite[Thm. 7.3]{Olshanskii}. Hence the group $H$ is artinian. 

For the in particular statement, we consider the sequence
$$\left\{1\right\}\to S\to S\times T\to (S\times T)/S\to\left\{1\right\}$$ 
that is exact and the torsion subgroups of $S$ and $(S\times T)/S\cong T$ are artinian. Thus the torsion subgroups of $S\times T$ are also artinian by the previous statement.
\end{proof}

\begin{proposition}\label{Extensions}
The class $\mathcal{G}$ is closed by taking extensions.
\end{proposition}
\begin{proof}
Let $\left\{1\right\}\rightarrow A\overset{\iota}{\rightarrow} B\overset{\pi}{\rightarrow} C\rightarrow \left\{1\right\}$ be an exact sequence of groups where $A$ and $C$ are in the class $\mathcal{G}$. Our goal is to show that $B$ is also in the class $\mathcal{G}$.

Suppose that $\mathbb{Q}$ is a subgroup of $B$. Since $\mathbb{Q}$ is divisible, the image of $\mathbb{Q}$ under $\pi$ is also an abelian divisible group. By assumption the group $C$ is contained in the class $\mathcal{G}$ and therefore any abelian divisible subgroup of $C$ is trivial by Theorem \ref{structureTheorem}. Hence $\mathbb{Q}\subseteq \ker(\pi)=\iota(A)\cong A$, this contradicts the fact that the group $A$ is in the class $\mathcal{G}$. The same arguments hold for the group $\mathbb{Z}(p^\infty$), since this group is also divisible.

Suppose now that $\mathbb{Z}_p$ is a subgroup of $B$. The kernel of $\pi_{\mid\mathbb{Z}_p}\colon\mathbb{Z}_p\to C$ is non-trivial since the group $C$ is in the class $\mathcal{G}$. Thus, by Lemma \ref{QuotientZp} we know that $\pi(\mathbb{Z}_p)\cong T\times D$ where $T$ is a finite cyclic group of $p$-power order and $D$ is an abelian divisible group. By assumption the group $C$ is in the class $\mathcal{G}$ and therefore $D$ has to be trivial. Thus $\ker(\pi_{\mid\mathbb{Z}_p})\subseteq \ker(\pi)=\iota(A)\cong A$ is an open subgroup of $p$-power index of $\mathbb{Z}_p$ by \cite[Thm. 5.2]{Klopsch}. It is therefore also a closed subgroup and thus is isomorphic to $\mathbb{Z}_p$ by \cite[Prop. 2.7(b)]{Profinite}, which is a contradiction to the assumption that $A$ is in the class $\mathcal{G}$.

By assumption the torsion subgroups of $A$ and $C$ are artinian, thus by Lemma \ref{artinian} the torsion subgroups of $B$ are artinian too. Hence the group $B$ is in the class $\mathcal{G}$.
\end{proof}

Now we show that the class $\mathcal{G}$ is closed under taking graph products. We want to do this with geometric means. This is inspired by \cite{KramerVarghese}.

\begin{definition}
Given a finite simplicial graph $\Gamma=(V, E)$ and a collection of groups $\mathcal{G} = \{ G_u \mid u \in V\}$, the \emph{graph product} $G_\Gamma$ is defined as the quotient
$$({\ast}_{u\in V} G_u) / \langle \langle [G_v,G_w]\text{ for }\{v,w\}\in E \rangle \rangle.$$
\end{definition}
	Given a graph product $G_\Gamma$, there exists a finite dimensional right-angled building $X_\Gamma$ on which $G_\Gamma$ acts isometrically \cite[Thm. 5.1]{DavisBuildings}, see also \cite[Section 3.5]{KramerVarghese}. \begin{proposition}\label{Building}
 If a subgroup $H\leq G_\Gamma$ of a graph product $G_\Gamma$ acts locally elliptically on the associated building $X_\Gamma$, i.e. each element has a fixed point, then $H$ has a global fixed point and $H$ is contained in a point-stabilizer, which has the form $gG_\Delta g^{-1}$ for a maximal clique $\Delta$ in $\Gamma$ and a $g\in G_\Gamma$.
\end{proposition}
\begin{proof}
	This can be found in \cite[Section 3.5]{KramerVarghese} and \cite[Lemma 3.6]{KramerVarghese}.
\end{proof}

We are now ready to prove our final proposition.
\begin{proposition}\label{ClassG}
	The class $\mathcal{G}$ is closed under taking graph products.
\end{proposition}
In particular any finite direct or free product of the groups named in Corollary C provides another example to which Theorem B can be applied.
\begin{proof}
	Given a  graph product $G_\Gamma$, let $X_\Gamma$ denote the associated finite dimensional right-angled building. Such a building is a CAT$(0)$ cube complex \cite[Thm. 5.1, Thm. 11.1]{DavisBuildings} and therefore we can apply  \cite{Bridson} to see that every isometry of $X_\Gamma$ is semi-simple and that the infimum of translation lengths of hyperbolic isometries is positive.
	
	So now suppose $\mathbb{Q}$ is not a subgroup of all the vertex groups $G_v$ but a subgroup of $G_\Gamma$. Then $\mathbb{Q}$ acts on $X_\Gamma$ via semi-simple isometries. Since $X_\Gamma$ is CAT$(0)$ we can apply \cite[Thm 2.5, Claim 7]{Caprace Marquis} to see that the action of $\mathbb{Q}$ has to be locally elliptic (since the infimum of translation lengths of hyperbolic isometries is positive). Therefore we can apply Proposition \ref{Building} to see that $\mathbb{Q}$ is contained in a vertex stabilizer. So there exists a complete subgraph $\Delta$ such that $\mathbb{Q}\subseteq G_\Delta=G_{v_1}\times G_{v_2}\times...\times G_{v_n}$ for some $n\in\mathbb{N}$, because stabilizers are conjugates of $G_\Delta$ and conjugating is an isomorphism of groups. We can then obtain maps $\pi_i\colon \mathbb{Q}\to G_{v_i}$ for $i\in\{1,...,n\}$ by taking quotients, $\pi_i$ is the canonical quotient map $G_\Delta\to G_\Delta\big/\left(G_{v_1}\times...\times G_{v_{i-1}}\times G_{v_{i+1}}\times...\times G_{v_n}\right)\cong G_{v_i}$. Each $\pi_i(\mathbb{Q})$ needs to be an abelian divisible subgroup of $G_i$ and hence needs to be trivial by Theorem \ref{structureTheorem}. This cannot be the case however, so the assumption that $G_\Gamma$ contains $\mathbb{Q}$ was wrong.\\
	The same argument holds for the $p$-Pr\"ufer groups.\\
	For $\mathbb{Z}_p$ we apply a similar argument to conclude that $\mathbb{Z}_p\subseteq G_\Delta=G_{v_1}\times G_{v_2}\times...\times G_{v_n}$ for some $n\in\mathbb{N}$. This is possible, since $\mathbb{Z}_p$ is $q$-divisible for every prime $q\neq p$. We again obtain maps $\pi_i\colon \mathbb{Z}_p\to G_{v_i}$ for $i\in\{1,...,n\}$. Since no $G_i$ contains $\mathbb{Z}_p$ each $\pi_i(\mathbb{Z}_p)$ needs to be a proper quotient of $\mathbb{Z}_p$, thus there are finite groups $T_i$ and abelian divisible groups $A_i$ with $\pi_i(\mathbb{Z}_p)\cong T_i\times A_i$ by Lemma \ref{QuotientZp}. But as above, $A_i$ has to be trivial. But this is a contradiction, since $\mathbb{Z}_p$ is not finite.
	
	Finally to show that torsion subgroups of $G_\Gamma$ are artinian we reduce to the direct product case as follows: Given a torsion subgroup $H$ of $G_\Gamma$, we know it needs to act locally elliptically on $G_\Gamma$, since every element has finite order. Thus due to Proposition \ref{Building} the torsion group $H$ is contained in $ g (G_{v_1}\times ... \times G_{v_k}) g^{-1}$ for some vertex groups $G_{v_i}$ for some $k\in \mathbb{N}$, $1\leq i\leq k$ and some $g\in G$. So $ G_{v_1}\times ... \times G_{v_k}$ contains a subgroup isomorphic to $H$. Since all the torsion subgroups in the vertex groups are artinian, we can apply Lemma \ref{artinian} to see that the torsion subgroups of $G_{v_1}\times ... \times G_{v_k}$ are artinian too. Therefore $H$ is artinian too, which is what we wanted to show.
\end{proof}

\end{document}